\documentclass[a4paper,reqno,10pt]{amsart}

\usepackage{hyperref}
\usepackage{a4wide}

\usepackage{comment}
\usepackage{amssymb}
\usepackage{amstext}
\usepackage{amsmath}
\usepackage{amscd}
\usepackage{amsthm}
\usepackage{amsfonts}

\usepackage{graphicx}
\usepackage{latexsym}
\usepackage{mathrsfs}
\usepackage{longtable}
\usepackage{url}

\usepackage{enumitem}
\setlist[enumerate,1]{label={\upshape(\arabic*)}}
\setlist[enumerate,2]{label={\upshape(\alph*)}}

\numberwithin{table}{section}

\usepackage{tikz}
\usetikzlibrary{cd,arrows,matrix,backgrounds,shapes,positioning,calc,decorations.markings,decorations.pathmorphing,decorations.pathreplacing}
\tikzset{blackv/.style={circle,fill=black,inner sep=3pt,outer sep=3pt},
         whitev/.style={circle,fill=white,draw=black,inner sep=3pt,outer sep=3pt},
         blabel/.style={circle,draw=black,inner sep=1.5pt,outer sep=0pt},
         redv/.style={circle,fill=red,inner sep=3pt,outer sep=3pt},
         bluev/.style={circle,fill=blue,inner sep=3pt,outer sep=3pt},
         block/.style={draw,rectangle split,rectangle split horizontal,rectangle split parts=#1},
         symbol/.style={
           draw=none,
           every to/.append style={
             edge node={node [sloped, allow upside down, auto=false]{$#1$}}}}
}

\usepackage{array}
\newcolumntype{C}{>{$}c<{$}}
\newcolumntype{x}[1]{>{\centering\arraybackslash\hspace{0pt}}m{#1}}
\usepackage{makecell}

\newtheorem{theorem}{Theorem}[section]
\newtheorem{theoremi}{Theorem}

\newtheorem{corollary}[theorem]{Corollary}

\newtheorem*{lemma*}{Lemma}
\newtheorem*{theorem*}{Theorem}
\newtheorem{proposition}[theorem]{Proposition}
\newtheorem{definition-proposition}[theorem]{Definition-Proposition}

\theoremstyle{definition}
\newtheorem{definition}[theorem]{Definition}

\newtheorem{remark}[theorem]{Remark}
\newtheorem{example}[theorem]{Example}

\newtheorem*{ack}{Acknowledgments}

\newcommand{\al}{\alpha}
\newcommand{\be}{\beta}

\renewcommand{\AA}{\mathcal{A}}

\newcommand{\CC}{\mathcal{C}}

\newcommand{\DD}{\mathcal{D}}

\newcommand{\HH}{\mathcal{H}}

\newcommand{\TT}{\mathcal{T}}

\newcommand{\UU}{\mathcal{U}}
\newcommand{\VV}{\mathcal{V}}
\newcommand{\WW}{\mathcal{W}}

\newcommand{\Z}{\mathbb{Z}}

\newcommand{\RHom}{\mathbf{R}\strut\kern-.2em\operatorname{Hom}\nolimits}

\newcommand{\Image}{\operatorname{Im}\nolimits}
\newcommand{\Kernel}{\operatorname{Ker}\nolimits}
\newcommand{\Cokernel}{\operatorname{Coker}\nolimits}

\newcommand{\coker}{\Cokernel}
\newcommand{\im}{\Image}
\renewcommand{\ker}{\Kernel}

\newcommand{\se}{\subseteq}

\newcommand{\ot}{\leftarrow}

\DeclareMathOperator{\moduleCategory}{\mathsf{mod}} \renewcommand{\mod}{\moduleCategory}

\DeclareMathOperator{\tors}{\mathsf{tors}}

\DeclareMathOperator{\Hasse}{\mathsf{Hasse}}

\DeclareMathOperator{\Fac}{\mathsf{Fac}}

\newcommand{\iso}{\cong}

\newcommand{\sst}[1]{\substack{#1}}

\newcommand{\BB}{\mathcal{B}}
\newcommand{\EE}{\mathcal{E}}

\numberwithin{equation}{section}

\begin{document}
\title[Classifying $t$-structures via  ICE-closed subcategories]{Classifying $t$-structures via  ICE-closed subcategories and a lattice of torsion classes}

\author[A. Sakai]{Arashi Sakai}
\address{A. Sakai: Graduate School of Mathematics, Nagoya University, Chikusa-ku, Nagoya, 464-8602, Japan}
\email{m20019b@math.nagoya-u.ac.jp}


\begin{abstract}
In a triangulated category equipped with a $t$-structure, we investigate a relation between ICE-closed (=Image-Cokernel-Extension-closed) subcategories of the heart of the $t$-structure and aisles in the triangulated categories. We introduce an ICE sequence, a sequence of ICE-closed subcategories satisfying a certain condition, and establish a bijection between ICE sequences and homology-determined preaisles. Moreover we give a sufficient condition that an ICE sequence induces a $t$-structure via the bijection. In the case of the bounded derived category $D^b(\mod\Lambda)$ of a $\tau$-tilting finite algebra $\Lambda$, we give a description of ICE sequences in $\mod\Lambda$ which induce bounded $t$-structures on $D^b(\mod\Lambda)$ from the viewpoint of a lattice consisting of torsion classes in $\mod\Lambda$.
\end{abstract}

\maketitle
\tableofcontents

\section{Introduction}

In representation theory of finite dimensional algebras, there are two important classes of subcategories, \emph{torsion classes} and \emph{wide subcategories}. The notion of a \emph{torsion pair} was introduced in \cite{Dic}, and is studied in relation with tilting or $\tau$-tilting theory \cite{BB,AIR,DIJ}. On the other hand, wide subcategories, abelian subcategories closed under extensions are investigated in connection with torsion classes and ring epimorphisms \cite{IT,MS,AP,AS}. 

Recently, in \cite{Eno}, torsion classes and wide subcategories are unified by the notion of \emph{ICE-closed subcategories}, subcategories closed under taking images, cokernels and extensions. In \cite{ES}, it is shown that ICE-closed subcategories of an abelian length category are precisely torsion classes in some wide subcategories of the abelian category. 

In the derived category of an abelian category, torsion pairs in the abelian category bijectively correspond to \emph{intermediate $t$-structures} \cite{HRS,BR,Woo}. Besides wide subcategories of the abelian category bijectively correspond to \emph{$H^0$-stable thick subcategories} \cite{Bru,ZC}. It is natural to ask what objects in the derived category correspond to ICE-closed subcategories of the abelian category.

In this paper, we observe the above correspondences via sequences of ICE-closed subcategories. The notion of preaisles, subcategories of a triangulated categories closed under extensions and positive shifts unifies aisles of $t$-structures and thick subcategories. We consider \emph{homology-determined} preaisles, a generalization of $H^0$-stable thick subcategories. In \cite{SR}, Stanley and van Roosmalen classify homology-determined preaisles by narrow sequences (Definition \ref{def:narseq}). We introduce the concept of an \emph{ICE sequence}, a sequence of ICE-closed subcategories of an abelian category satisfying some conditions (Definition \ref{def:ICEseq}). Then we show that ICE sequences coincide with narrow sequences in Proposition \ref{prop:char}, and establish the following correspondences.

\begin{theoremi}[Theorem \ref{thm:ICE}]\label{thm:1}
Let $\DD$ be a triangulated category and $(\DD^{\leq 0},\DD^{\geq 1})$ a $t$-structure on $\DD$ with the heart $\HH$. 
Then there exists a bijective correspondence between
\begin{enumerate}
\item the set of ICE sequences in $\HH$,
\item the set of homology-determined preaisles in $\DD$. 
\end{enumerate}
\end{theoremi}

As a corollary of the above, we recover main results in \cite{Bru,ZC} (Corollary \ref{cor:H0}) and obtain an analogous result in \cite{HRS,BR,Woo} (Corollary \ref{cor:HRS}). 

Next we focus on aisles, equivalent to $t$-structures \cite{KV}. There are many results classifying t-structures on the bounded derived category, e.g. a finite dimensional hereditary algebra \cite{KV,SR,ST}, a finite dimensional algebra \cite{KY}, a commutative ring \cite{AJS} and coherent sheaves on the projective line \cite{GKR,SR}, and on the unbounded derived category of a ring \cite{AMV,AH}. 

We aim to classify homology-determined aisles. In Theorem \ref{thm:1}, a homology-determined aisle in $\DD$ corresponds to an ICE sequence of $\HH$ consisting of \emph{contravariantly finite} ICE-closed subcategories of $\HH$ (Proposition \ref{prop:nece}). However we do not know whether the converse holds or not. We assume the conditions that the $t$-structure $(\DD^{\leq 0},\DD^{\geq 1})$ is bounded and that the heart $\HH$ is contravariantly finite in $\DD$, which is studied in \cite{CPP}. For example, the condition is fulfilled where $\DD=D^b(\mod\Lambda)$ for a finite dimensional algebra $\Lambda$ and $(\DD^{\leq 0},\DD^{\geq 1})$ is the standard $t$-structure. Under the assumption, an ICE sequence which \emph{terminates to zero} induces an aisle (Proposition \ref{prop:term}). Then we obtain a classification of bounded $t$-structures on the bounded derived category $D^b(\mod\Lambda)$ of a finite dimensional algebra $\Lambda$ whose aisles are homology-determined.

\begin{theoremi}[Theorem \ref{thm:bdd}]\label{thm:2}
Let $\Lambda$ be a finite dimensional algebra.
Then there exists a bijective correspondense between
\begin{enumerate}
\item the set of contravariantly finite full ICE sequences in $\mod\Lambda$,
\item the set of bounded $t$-structures on $D^b(\mod\Lambda)$ whose aisles are homology-determined. 
\end{enumerate}
\end{theoremi}

We give a description of full ICE sequences in $\mod\Lambda$ from the viewpoint of a lattice $\tors\Lambda$ consisting of torsion classes in $\mod\Lambda$. There are many studies of lattices of torsion classes \cite{AIR,DIRRT,Eno2,AP}. In \cite{AP}, Asai and Pfeifer introduce the notion of \emph{wide intervals} and \emph{meet intervals} and show that they are equivalent, see Proposition \ref{prop:torslattice} (1). We introduce a \emph{maximal meet interval} (Definition \ref{def:max}), which is defined by only a lattice-theoretical property of $\tors\Lambda$. We consider a \emph{decreasing sequence of maximal meet intervals} and obtain the following classification. 

\begin{theoremi}[Theorem \ref{thm:class}]\label{thm:3}
Let $\Lambda$ be a $\tau$-tilting finite algebra and $n$ a positive integer.
Then there are one-to-one correspondenses between
\begin{enumerate}
\item the set of $(n+1)$-intermediate $t$-structures on $D^b(\mod\Lambda)$ whose aisles are homology-determined,
\item the set of ICE sequence in $\mod\Lambda$ of length $n+1$,
\item the set of decreasing sequences of maximal meet intervals in $\tors\Lambda$ of length $n$.
\end{enumerate}
\end{theoremi}

\medskip
\noindent
{\bf Organization.}
In Sect.\ref{sec:pre}, we collect basic definitions. In Sect.\ref{sec:nar}, we recall the classification of homology-determined preaisles via narrow sequences. In Sect.\ref{sec:ICE}, we introduce ICE sequences and give a proof of Theorem \ref{thm:1}. Moreover we consider restrictions of it to thick subcategories and $n$-intermediate preaisles. In Sect.\ref{sec:cla}, we give a proof of the main results Theorems \ref{thm:2} and \ref{thm:3} and show some examples. 

\medskip
\noindent
{\bf Conventions and notation.}
Throughout this paper, we denote by $\AA$ an abelian category.
We denote by $\DD$ a triangulated category with shift functor $[1]$.
We assume that all subcategories are full, additive, and closed under isomorphisms. For a subcategory $\CC$ of $\AA$, we put
\[
 {}^\perp\CC=\{X\in\AA\mid\AA(X,C)=0 \ \text{for any} \ C\in\CC\},
\]
\[
 \CC^{\perp}=\{X\in\AA\mid\AA(C,X)=0 \ \text{for any} \ C\in\CC\}.
\]

\section{Preliminaries}\label{sec:pre}

We start introducing basic definitions.

\begin{definition}\label{def:basicdef}
  Let $\CC$ be a subcategory of $\AA$.
  \begin{enumerate}
    \item $\CC$ is \emph{closed under extensions} if, for any short exact sequence 
    \[
    \begin{tikzcd}
      0 \rar & L \rar & M \rar & N \rar & 0
    \end{tikzcd}
    \]
    in $\AA$, we have that $L,N \in \CC$ implies $M \in \CC$.
    \item $\CC$ is \emph{closed under quotients in $\AA$} if, for every object $C \in \CC$, any quotients of $C$ in $\AA$ belong to $\CC$.
    \item $\CC$ is a \emph{torsion class in $\AA$} if $\CC$ is closed under extensions and quotients in $\AA$.
    \item $\CC$ is closed under \emph{images (resp. kernels, cokernels)} if, for every morphism $\varphi \colon C_1 \to C_2$ with $C_1, C_2 \in \CC$, we have $\im\varphi \in \CC$ (resp. $\ker\varphi\in\CC$, $\coker\varphi\in\CC$).
    \item $\CC$ is a \emph{wide subcategory of $\AA$} if $\CC$ is closed under extensions, kernels and cokernels.
    \item $\CC$ is an \emph{ICE-closed subcategory of $\AA$} if $\CC$ is closed under images, cokernels and extensions.
  \end{enumerate}
\end{definition}

Every wide subcategory $\WW$ of $\AA$ becomes an abelian category, and we always regard $\WW$ as such.
Both wide subcategories and torsion classes are ICE-closed subcategories. Moreover, every torsion class in a wide subcategory (viewed as an abelian category) is ICE-closed, see \cite[Lemma 2.2]{ES}. In \cite{IT}, Ingalls and Thomas introduced an operation $\alpha$ which associates to a torsion class a wide subcategory. In \cite{Eno}, the operation was generalized to ICE-closed subcategories. For the convenience of the reader, we give a direct proof that $\alpha\CC$ is a wide subcategory for any ICE-closed subcategory $\CC$ of $\AA$. This is the same argument of \cite{IT}.

\begin{proposition}\cite[Proposition 2.12]{IT}\cite[Proposition 4.2]{Eno}\label{prop:alpha}
Let $\CC$ be an ICE-closed subcategory of $\AA$. Define a subcategory of $\CC$ by
\[
 \alpha\CC=\{A\in\CC\mid {}^{\forall}(f\colon C\to A)\in\CC, \ \ker f\in\CC\}.
\]
Then the following statements hold.
\begin{enumerate}
\item Let $C$ be an object in $\alpha\CC$ and $X$ a subobject of $C$ which belongs to $\CC$. Then $X$ belongs to $\alpha\CC$.
\item $\alpha\CC$ is a wide subcategory of $\AA$. 
\end{enumerate}
\end{proposition}

\begin{proof}
(1) Let $f\colon C'\to X$ be a morphism with $C'\in\CC$. Then $\ker f$ coincides with the kernel of the composition of $f$ and the monomorphism $X\to C$. By $C\in\alpha\CC$ and $C'\in\CC$, we have $\ker f\in\CC$ as desired. 

(2) We show that $\alpha\CC$ is closed under extensions. Let $0\to X\to Y\to Z\to 0$ be a short exact sequence in $\AA$ with $X,Z\in\alpha\CC$. Since $\CC$ is closed under extensions, we have $Y\in\CC$. Take a morphism $f\colon C\to Y$ with $C\in\CC$. We show $\ker f\in\CC$. Consider the following exact diagram:
  \[
  \begin{tikzcd}
    & 0 \dar & 0 \dar \\
    & \ker f \rar[equal] \dar & \ker f \dar \\
    0 \rar & K \rar\dar\ar[rd, phantom, "{\rm p.b.}"] & C \rar \dar["f"] & Z \dar[equal] & \ \\
    0 \rar & X \rar & Y \rar & Z \rar & 0
  \end{tikzcd}
  \]
where $K$ is the kernel of the morphism $C\to Z$ and the left lower square is a pullback diagram. Then $K$ belongs to $\CC$ by $Z\in\alpha\CC$. Since $X$ belongs to $\alpha\CC$, we have $\ker f\in\CC$ as desired.

We show that $\alpha\CC$ is closed under cokernels. Let $f\colon X\to Y$ be a morphism with $X,Y\in\alpha\CC$. Then we have $\coker f\in\CC$ since $\CC$ is closed under cokernels. Take a morphism $g\colon C\to \coker f$ with $C\in\CC$. We show $\ker g\in\CC$. Taking a pullback, we have the following exact diagram:
  \[
  \begin{tikzcd}
    & \ & 0 \dar & 0\dar \\
    & \ & \ker g \dar\rar[equal] & \ker g \dar \\
    0 \rar & \im f \rar\dar[equal] & V \rar\dar\ar[rd, phantom, "{\rm p.b.}"] & C \rar\dar["g"] & 0 \\
    0 \rar & \im f \rar & Y \rar & \coker f \rar & 0.
  \end{tikzcd}
  \]
Since $\CC$ is closed under images, $\im f$ belongs to $\CC$. Then the middle row implies $V\in\CC$ since $\CC$ is closed under extensions. Therefore $\ker g$ belongs to $\CC$ since $Y$ belongs to $\alpha\CC$. 

We show that $\alpha\CC$ is closed under kernels. Let $f\colon X\to Y$ be a morphism with $X,Y\in\alpha\CC$. Then we have $\ker f\in\CC$ since $Y$ belongs to $\alpha\CC$. Applying (1) to the morphism $\ker f\to X$, we have $\ker f\in\alpha\CC$ as desired. Thus $\alpha\CC$ is a wide subcategory of $\AA$.
\end{proof}

For subcategories $\UU$ and $\VV$ of $\DD$, we denote by $\UU*\VV$ the subcategory of $\DD$ consisting of objects $X$ such that there exists an exact triangle $U\to X\to V\to U[1]$ in $\DD$ with $U\in\UU$ and $V\in\VV$. A subcategory $\UU$ of $\DD$ is \emph{closed under extensions} if it satisfies $\UU*\UU\subseteq\UU$.

\begin{definition}\cite[D\'{e}finition 1.3.1]{BBD}
A pair of subcategories $(\DD^{\leq 0},\DD^{\geq 1})$ is a \emph{$t$-structure} on $\DD$ if it satisfies the following conditions:
\begin{enumerate}
\item $\DD(\DD^{\leq 0},\DD^{\geq 1})=0$.
\item $\DD=\DD^{\leq 0}*\DD^{\geq 1}$.
\item $\DD^{\leq 0}[1]\se\DD^{\leq 0}$.
\end{enumerate}
\end{definition}

We put $\DD^{\leq n}=\DD^{\leq 0}[-n]$ and $\DD^{\geq n+1}=\DD^{\geq 1}[-n]$. The subcategory $\HH=\DD^{\leq 0}\cap\DD^{\geq 0}$ is called the \emph{heart} of $(\DD^{\leq 0},\DD^{\geq 1})$, which is an abelian category. We denote by $H^0\colon\DD\to\HH$ the associated cohomological functor and put $H^k=H^0[k]$. We call $\DD^{\leq 0}$ an \emph{aisle} and $\DD^{\geq 1}$ a \emph{coaisle}. A $t$-structure $(\DD^{\leq 0},\DD^{\geq 1})$ is \emph{bounded} if it satisfies
\[
\bigcup_{n\in\Z}\DD^{\leq n}=\DD=\bigcup_{n\in\Z}\DD^{\geq n}
\]
and \emph{nondegenerate} if it satisfies
\[
\bigcap_{n\in\Z}\DD^{\leq n}=0=\bigcap_{n\in\Z}\DD^{\geq n}.
\]
It is easily checked that bounded $t$-structures are nondegenerate. A bounded $t$-structure is \emph{algebraic} if its heart is an abelian length category with finitely many simple objects up to isomorphisms. The following proposition is useful and shown straightforwardly.

\begin{proposition}\label{prop:decomp}
Let $X\in \DD^{\leq n}\cap \DD^{\geq m}$ with $m\leq n$. Then there exist exact triangles
\[
X_{i-1}\to X_i\to (H^iX)[-i]\to X_{i-1}[1]
\]
for $m\leq i\leq n$ where $X_n=X$ and $X_{m-1}=0$. 
\end{proposition}

Let $\CC$ be an additive category and $\BB$ a subcategory of $\CC$. A morphism $f\colon B\to X$ in $\CC$ is a \emph{right $\BB$-approximation ($\BB$-precover)} of $X$ if $B$ belongs to $\BB$ and every morphism $f'\colon B'\to X$ with $B'\in\BB$ factors through $f$. Dually, a \emph{left $\BB$-approximation ($\BB$-preenvelope)} is defined. A subcategory $\BB$ of $\CC$ is \emph{contravariantly finite (resp. covariantly finite)} if every $X\in\CC$ has a right (resp. left) $\BB$-approximation. The notion of contravariantly (resp. covariantly) finite is also called \emph{precovering (resp. preenveloping)}. The term contravariantly (resp. covariantly) finite is also referred to as \emph{precovering (resp. preenveloping)}. A subcategory of $\CC$ is \emph{functorially finite} if it is contravariantly finite and covariantly finite. 

\section{Narrow sequences and homology-determined preaisles}\label{sec:nar}

In this section, we recall the result \cite[Theorem 4.11]{SR}. It is stated that homology-determined preaisles in the bounded derived category of an abelian category are classified by narrow sequences in the abelian category. We generalize it to a triangulated category equipped with a $t$-structure (Theorem \ref{thm:SR}). The proof is dealt with a similar argument of \cite{SR}. For the convenience of the reader, we give the proof. In the rest of this paper, we fix a $t$-structure $(\DD^{\leq 0},\DD^{\geq 1})$ on $\DD$. We start by giving the definition of narrow sequences. 

\begin{definition}\cite[Definition 4.1]{SR}\label{def:narseq}
A sequence $\{\CC(k)\}_{k\in\Z}$ of subcategories of an abelian category $\AA$ satisfying $\CC(k+1)\subseteq\CC(k)$ for any $k\in\Z$ is a \emph{narrow sequence} if for any $k\in\Z$, the following holds: Let 
\begin{equation}\label{eq:narseq}
\begin{tikzcd}
  A \rar["a"] & B \rar["b"] & C \rar["c"] & D\rar["d"] & E
\end{tikzcd}
\end{equation}
be an exact sequence in $\AA$ with $A\in\CC(k-1)$ and $B, D\in\CC(k)$ and $E\in\CC(k+1)$. Then we have $C\in\CC(k)$. 
\end{definition}

\begin{remark}
In \cite[Lemma 4.2]{SR}, it is shown that for each $k$, the subcategory $\CC(k)$ is a \emph{narrow subcategory}, subcategory closed under extensions and cokernels. Moreover it is stated in \cite[Corollary 3.3]{SR} that narrow subcategories of a hereditary abelian category are closed under images, hence are ICE-closed subcategories. In Proposition \ref{prop:ICEseq}, we show that $\CC(k)$ is an ICE-closed subcategory without the assumption that $\AA$ is hereditary.
\end{remark}

The following statement gives a characterization of a narrow sequence.

\begin{proposition}\label{prop:narrow}
Let $\{\CC(k)\}_{k\in\Z}$ be a sequence of subcategories of $\AA$ satisfying $\CC(k+1)\subseteq\CC(k)$ for any $k\in\Z$. Then $\{\CC(k)\}_{k\in\Z}$ is a narrow sequence if and only if it satisfies the following:
\begin{enumerate}
\item $\CC(k)$ is closed under extensions in $\AA$ for any $k\in\Z$. 
\item For any $k\in\Z$ and any morphism $f\colon A\to B$ with $A\in\CC(k)$ and $B\in\CC(k+1)$, we have $\ker f\in\CC(k)$ and $\coker f\in\CC(k+1)$.
\end{enumerate}
\end{proposition}

\begin{proof}
Suppose that $\{\CC(k)\}_{k\in\Z}$ is a narrow sequence. Considering (\ref{eq:narseq}) with $A=E=0$, we have that $\CC(k)$ is closed under extensions. Let $f\colon X\to Y$ be a morphism with $X\in\CC(k)$ and $Y\in\CC(k+1)$. By considering (\ref{eq:narseq}) with $d=f$ and $A=B=0$, we have $\ker f\in\CC(k)$. Similarly, it is shown that $\coker f$ belongs to $\CC(k+1)$. 

Conversely, suppose that $\{\CC(k)\}_{k\in\Z}$ satisfies (1) and (2). Consider the sequence (\ref{eq:narseq}). By (2), we have that $\coker a$ and $\ker d$ belong to $\CC(k)$. Consider the following short exact sequence: 
\[
\begin{tikzcd}
  0 \rar & \coker a \rar & C \rar & \ker d\rar & 0.
\end{tikzcd}
\]
By (1), we have $C\in\CC(k)$ as desired.
\end{proof}

Note that the condition {\rm (2)} appears in \cite[Proposition 5.1]{AH}. A subcategory $\UU$ of $\DD$ is a \emph{preaisle} if it is closed under extensions and positive shifts in $\DD$. 

\begin{definition}\cite[Definition 4.5]{SR}
A subcategory $\UU$ of $\DD$ is \emph{homology-determined} if for any $X\in\DD$, we have that $X\in\UU$ if and only if $H^kX\in \UU[k] \ \text{for any} \ k\in\Z$. 
\end{definition}

Note that the above definition depends on the fixed $t$-structure $(\DD^{\leq 0}, \DD^{\geq 1})$. It is easy to check that if $\UU$ is homology-determined, then so is $\UU[k]$ for any $k\in\Z$.

\begin{remark}\label{rem:homdet}
\begin{enumerate}
\item The term homology-determined is also called \emph{determined on cohomology} in \cite{AH}. 
\item If $(\DD^{\leq 0},\DD^{\geq 1})$ is bounded, then a subcategory $\UU$ of $\DD$ closed under extensions is homology-determined if and only if it satisfies the following condition: For any $X\in\UU$ and any $k\in\Z$, it holds $H^kX\in\UU[k]$. This follows from Proposition \ref{prop:decomp}. 
\item In the bounded derived category of a hereditary abelian category with the standard $t$-structure, any subcategory closed under extensions and direct summands is homology-determined. 
\item In general, a coaisle is not always homology-determined even if the corresponding aisle is homology-determined, see Remark \ref{rem:coaisle}.
\end{enumerate}
\end{remark}

\begin{proposition}\label{prop:homdet}
Let $\UU$ be a homology-determined preaisle of $\DD$. Then $H^k\UU=\UU[k]\cap\HH$ holds. In particular, $H^k\UU$ is closed under extensions in $\HH$.
\end{proposition}

\begin{proof}
We show $H^k\UU\se\UU[k]\cap\HH$. Clearly $H^k\UU\se\HH$ holds. Since $\UU$ is homology-determined, we have $H^k\UU\se\UU[k]$. Then we have $H^k\UU\se\UU[k]\cap\HH$. We show the converse. Let $X$ be an object in $\UU[k]\cap\HH$. Then there is an object $U$ of $\UU$ such that $X\iso U[k]$. Since $X\in\HH$, we have $H^0X\iso X$. Thus we have $X\iso H^0(U[k])=H^kU\in H^k\UU$ as desired. Since $\UU$ and $\HH$ are closed under extensions in $\DD$, so is $H^k\UU$. Thus $H^k\UU$ is closed under extensions in $\HH$. 
\end{proof}

A narrow sequence in $\HH$ induces a homology-determined preaisle in $\DD$ as follows:

\begin{proposition}\cite[Proposition 4.10]{SR}\label{prop:nartoaisle}
Let $\CC=\{\CC(k)\}_{k\in\Z}$ be a narrow sequence in $\HH$. Define a subcategory $\theta(\CC)$ of $\DD$ by
\[
\theta(\CC)=\{X\in \DD\mid H^kX\in\CC(k) \ \text{for any} \ k\in\Z\}.
\]
Then $\theta(\CC)$ is a homology-determined preaisle in $\DD$.
\end{proposition}

\begin{proof}
(1) We show that $\theta(\CC)$ is closed under positive shifts. Let $X$ be an object of $\theta(\CC)$. Then $H^k(X[1])=H^{k+1}X\in\CC(k+1)\se\CC(k)$ holds for any $k\in\Z$. Hence $X[1]$ belongs to $\theta(\CC)$. We show that $\theta(\CC)$ is closed under extensions. Let $X\to Y\to Z\to X[1]$ be an exact triangle in $\DD$ with $X,Z\in\theta(\CC)$. Taking cohomologies, we have the following exact sequence
\[
\begin{tikzcd}
  H^{k-1}Z \rar & H^{k}X \rar & H^{k}Y\rar & H^{k}Z\rar & H^{k+1}X
\end{tikzcd}
\]
in $\HH$ for any $k\in\Z$. By Definition \ref{def:narseq}, it is easy to see $H^kY\in\CC(k)$. Thus $Y$ belongs to $\theta(\CC)$. Therefore $\theta(\CC)$ is a preaisle. We show that $\theta(\CC)$ is homology-determined. Let $X$ be an object in $\DD$. It is enough to show that for any $k\in\Z$, we have $H^kX\in\CC(k)$ if and only if $(H^kX)[-k]\in\theta(\CC)$ holds. This follows easily from the definition of $\theta(\CC)$. 
\end{proof}

Conversely, a homology-determined preaisle in $\DD$ induces a narrow sequence in $\HH$ as follows:

\begin{proposition}\cite[Proposition 4.8]{SR}\label{prop:aisletonar}
Let $\UU$ be a homology-determined preaisle in $\DD$. Define a sequence $\mu(\UU)=\{\CC(k)\}_{k\in\Z}$ of subcategories of $\HH$ by $\CC(k)=H^{k}\UU$. Then $\mu(\UU)$ is a narrow sequence in $\HH$.
\end{proposition}

\begin{proof}
Since $\UU$ is closed under positive shifts, we have $\CC(k+1)=H^{k+1}\UU=H^{k}(\UU[1])\se H^k\UU=\CC(k)$ for any $k$. We show that it satisfies (1) and (2) in Proposition \ref{prop:narrow}. By Proposition \ref{prop:homdet}, we have that $\CC(k)$ is closed under extensions. Therefore it satisfies (1) in Proposition \ref{prop:narrow}. We show that it satisfies (2) in Proposition \ref{prop:narrow}. Let $f\colon A\to B$ be an arbitrary morphism in $\HH$ with $A\in\CC(k)$ and $B\in\CC(k+1)$. We show $\coker f\in\CC(k+1)$, similarly it is shown $\ker f\in\CC(k)$. Take an exact triangle 
\[
\begin{tikzcd}
  A[-k-1] \rar["{f[-k-1]}"] & B[-k-1] \rar & X\rar & A[-k]
\end{tikzcd}
\]
in $\DD$. Then both $B[-k-1]$ and $A[-k]$ belong to $\UU$ since $\UU$ is homology-determined. Therefore $X\in\UU$ since $\UU$ is closed under extensions. Applying $H^{k+1}$ to the above sequence, we have an exact sequence
\[
\begin{tikzcd}
  A \rar["f"] & B \rar & H^{k+1}X\rar & 0
\end{tikzcd}
\]
in $\HH$. Then we have $H^{k+1}X\iso\coker f$, and $H^{k+1}X\in\CC(k+1)$ holds since $X$ belongs to $\UU$. Thus we have $\coker f\in\CC(k+1)$ as desired. 
\end{proof}

The following is the main result in this section. 

\begin{theorem}\cite[Theorem 4.11]{SR}\label{thm:SR}
There exist mutually bijective correspondences between
\begin{enumerate}
\item the set of narrow sequences in $\HH$,
\item the set of homology-determined preaisles in $\DD$. 
\end{enumerate}
\end{theorem}

\begin{proof}
We show that the map $\theta$ and $\mu$ defined in Propositions \ref{prop:nartoaisle} and \ref{prop:aisletonar} are inverses of each other. Let $\CC=\{\CC(k)\}_{k\in\Z}$ be a narrow sequence in $\HH$. To show $\mu(\theta(\CC))=\CC$, it is enough to show $H^k(\theta(\CC))=\CC(k)$ for any $k\in\Z$. By definition of $\theta$, we have $H^k(\theta(\CC))\se\CC(k)$. Let $X$ be an object in $\CC(k)$. Then $X[-k]\in\theta(\CC)$ holds. Therefore $X\iso H^0X\iso H^k(X[-k])\in H^k(\theta(\CC))$ holds. 

Let $\UU$ be a homology-determined preaisle in $\DD$. We show $\theta(\mu(\UU))=\UU$. Let $X$ be an object in $\DD$. Since $\UU$ is homology-determined, $X\in\UU$ if and only if $H^kX\in\UU[k]$ for any $k\in\Z$. By Proposition \ref{prop:homdet}, we have $H^k\UU=\UU[k]\cap\HH$. Since $H^kX\in\HH$ holds, we have that $H^kX\in\UU[k]$ for any $k$ if and only if $H^kX\in H^k\UU$ for any $k$. This is equivalent to $X\in\theta(\mu(\UU))$. Thus we have done. 
\end{proof}

\begin{remark}\label{rem:summand}
By (2) in Proposition \ref{prop:narrow}, subcategories $\CC(k)$ in an ICE sequence $\{\CC(k)\}_{k\in\Z}$ are closed under cokernels in $\HH$, in particular direct summands in $\HH$. Then Theorem \ref{thm:SR} implies that any homology-determined preaisle in $\DD$ is closed under direct summands.
\end{remark}

\section{ICE sequences}\label{sec:ICE}

In this section, we introduce the notion of ICE sequences and observe that it is equivalent to narrow sequences. This characterization enables us to study preaisles, thereby $t$-structures via ICE-closed subcategories and a lattice of torsion classes.

\begin{definition}\label{def:ICEseq}
A sequence $\{\CC(k)\}_{k\in\Z}$ of subcategories of $\AA$ is an \emph{ICE sequence} if for any $k$, the subcategory $\CC(k)$ is an ICE-closed subcategory of $\AA$ and  the subcategory $\CC(k+1)$ is a torsion class in $\alpha(\CC(k))$.
\end{definition}

The above concept characterizes a narrow sequence:

\begin{proposition}\label{prop:char}
Let $\{\CC(k)\}_{k\in\Z}$ be a sequence of subcategories of $\AA$. Then $\{\CC(k)\}_{k\in\Z}$ is a narrow sequence in $\AA$ if and only if it is an ICE sequence in $\AA$.
\end{proposition}

We divide the proof of the above into the following two propositions.

\begin{proposition}\label{prop:ICEseq}
Let $\{\CC(k)\}_{k\in\Z}$ be a narrow sequence in $\AA$.
Then the following hold for any $k\in\Z$.
\begin{enumerate}
\item The subcategory $\CC(k)$ is an ICE-closed subcategory in $\AA$.
\item The subcategory $\CC(k+1)$ is a torsion class in $\alpha(\CC(k))$.
\end{enumerate}
\end{proposition}

\begin{proof}
(1) It is easy to verify that $\CC(k)$ is closed under extensions and cokernels by  Proposition \ref{prop:narrow}. We show that $\CC(k)$ is closed under images. Let $f\colon X\to Y$ be a morphism with $X,Y\in\CC(k)$. By Proposition \ref{prop:narrow} (2), we have $\ker f\in\CC(k-1)$. Applying (2) in Proposition \ref{prop:narrow} to the canonical morphism $\ker f\to X$, we have $\im f\in\CC(k)$. 

(2) By Proposition \ref{prop:narrow} (2), it is easy to verify $\CC(k+1)\se\alpha(\CC(k))$. We show that $\CC(k+1)$ is closed under quotients in $\alpha(\CC(k))$. Let 
\[
\begin{tikzcd}
  0 \rar & X \rar & Y \rar & Z\rar & 0
\end{tikzcd}
\]
be a short exact sequence in $\alpha(\CC(k))$ with $Y\in\CC(k+1)$. By $X\in\CC(k)$ and Proposition \ref{prop:narrow} (2), we have $Z\in\CC(k+1)$ as desired.
\end{proof}

\begin{proposition}\label{prop:ICEtonar}
An ICE sequence $\{\CC(k)\}_{k\in\Z}$ in $\AA$ is a narrow sequence in $\AA$. 
\end{proposition}

\begin{proof}
Clearly it satisfies $\CC(k+1)\se\alpha(\CC(k))\se\CC(k)$ and (1) in Proposition \ref{prop:narrow}. We show that it satisfies (2) in Proposition \ref{prop:narrow}. Let $f\colon X\to Y$ be a morphism with $X\in\CC(k)$ and $Y\in\CC(k+1)$. By $\CC(k+1)\se\alpha(\CC(k))$, we have $\ker f\in\CC(k)$. We show $\coker f\in\CC(k+1)$. Since $\CC(k)$ is closed under images, we may assume that $f$ is monomorphic. Then we have a short exact sequence 
\[
\begin{tikzcd}
  0 \rar & X \rar["f"] & Y \rar & \coker f\rar & 0
\end{tikzcd}
\]
in $\AA$. By $Y\in\alpha(\CC(k))$ and Proposition \ref{prop:alpha} (1), it holds $X\in\alpha(\CC(k))$. Then all terms in the short exact sequence belong to $\alpha(\CC(k))$. Since $\CC(k+1)$ is a torsion class in $\alpha(\CC(k))$ and $Y$ belongs to $\CC(k+1)$, we have $\coker f\in\CC(k+1)$ as desired. 
\end{proof}

\begin{proof}[Proof of Proposition \ref{prop:char}]
This follows from Propositions \ref{prop:ICEseq} and \ref{prop:ICEtonar}.
\end{proof}

Combining Theorem \ref{thm:SR} and Proposition \ref{prop:char}, we obtain the following result.

\begin{theorem}\label{thm:ICE}
There exist mutually bijective correspondences between
\begin{enumerate}
\item the set of ICE sequences in $\HH$,
\item the set of homology-determined preaisles in $\DD$. 
\end{enumerate}
\end{theorem}

Next we observe that the above recovers the main result of \cite{ZC}. A subcategory $\EE$ of $\DD$ is \emph{$H^0$-stable} if $H^0X\in\EE$ holds for any $X\in\EE$. 

\begin{proposition}\label{prop:H0}
Assume that $(\DD^{\leq 0},\DD^{\geq 1})$ is bounded.
A thick subcategory  $\EE$ of $\DD$ is $H^0$-stable if and only if it is homology-determined. 
\end{proposition}

\begin{proof}
Clearly, if $\EE$ is homology-determined, then it is $H^0$-stable. Assume that $\EE$ is $H^0$-stable. By Remark \ref{rem:homdet} (2), it is enough to show that for any $X\in\EE$ and any $k\in\Z$, it holds $H^kX\in\EE[k]$. This is easily checked since $\EE$ is closed under positive and negative shifts. 
\end{proof}

The following bijection between (1) and (3) is given in \cite{ZC}.

\begin{corollary}\label{cor:H0}
There exist bijective correspondences between:
\begin{enumerate}
\item the set of wide subcategories of $\HH$,
\item the set of homology-determined thick subcategories of $\DD$. 
\end{enumerate}
In the correspondence from {\rm (1)} to {\rm (2)}, a wide subcategory $\WW$ corresponds to a subcategory 
\[
\widetilde{\WW}=\{X\in\DD\mid {}^{\forall}k\in\Z, \ H^kX\in\WW\}
\]
and its inverse is given by $\EE\mapsto H^0\EE$. Moreover if $(\DD^{\leq 0},\DD^{\geq 1})$ is bounded, the set (2) equals to the following:
\begin{enumerate}
\item[{\rm (3)}] the set of $H^0$-stable thick subcategories of $\DD$. 
\end{enumerate}
\end{corollary}

\begin{proof}
Let $\UU$ be a homology-determined preaisle in $\DD$ and $\CC=\{\CC(k)\}_{k\in\Z}$ the corresponding ICE sequence in $\HH$. Then $\UU$ is closed under direct summands by Remark \ref{rem:summand}. It is easy to check that $\UU$ is a thick subcategory if and only if it holds $\CC(k)=\CC(k+1)$ for any $k\in\Z$. To a wide subcategory $\WW$, we associate an ICE sequence $\{\CC(k)\}_{k\in\Z}$ satisfying $\CC(k)=\WW$ for any $k\in\Z$. It is easy to verify that this gives a bijection between the set of wide subcategories of $\HH$ and the set of ICE sequences $\{\CC(k)\}_{k\in\Z}$ satisfying $\CC(k)=\CC(k+1)$ for any $k\in\Z$. Combining these, we obtain the bijection between (1) and (2). The latter part follows from Proposition \ref{prop:H0}.
\end{proof}

In the rest of this section, we consider ICE sequences which correspond to $n$-intermediate preaisles. 

\begin{definition}
Let $n$ be a positive integer. 
\begin{enumerate}
\item A preaisle $\UU$ in $\DD$ is \emph{$n$-intermediate} if it satisfies $\DD^{\leq -n+1}\subseteq\UU\subseteq \DD^{\leq 0}$. 
\item A sequence $\{\CC(k)\}_{k\in\Z}$ \emph{has length $n$} if it satisfies $\CC(1)=0$ and $\CC(-n+1)=\HH$.
\end{enumerate}
We call $2$-intermediate simply \emph{intermediate}.
\end{definition}

The correspondences in Theorem \ref{thm:ICE} restrict to $n$-intermediate preaisles as follows:

\begin{corollary}\label{cor:n-int}
Suppose that $(\DD^{\leq 0},\DD^{\geq 1})$ is nondegenerate. Then there exist bijective correspondences between
\begin{enumerate}
\item the set of ICE sequences in $\HH$ of length $n$,
\item the set of $n$-intermediate homology-determined preaisles in $\DD$. 
\end{enumerate}
\end{corollary}

\begin{proof}
We show the maps in Theorem \ref{thm:ICE} restrict to (1) and (2). Let $\CC=\{\CC(k)\}_{k\in\Z}$ be an ICE sequence of length $n$. We show $\DD^{\leq -n+1}\se\theta(\CC)$. Let $X$ be an object in $\DD^{\leq -n+1}$. Then $H^kX=0$ holds for any $k>-n+1$. Therefore $X\in\theta(\CC)$ holds by definition of $\theta(\CC)$. We show $\theta(\CC)\se\DD^{\leq 0}$. Let $X$ be an object in $\theta(\CC)$. By $\CC(1)=0$, it holds $H^kX=0$ for any $k>0$. Since $(\DD^{\leq 0},\DD^{\geq 1})$ is nondegenerate, we have $X\in\DD^{\leq 0}$. 

Let $\UU$ be an $n$-intermediate homology-determined preaisle in $\DD$. We show that $\{H^k\UU\}_{k\in\Z}$ is an ICE sequence of length $n$. Since $\UU\se\DD^{\leq 0}$, we have $H^1\UU=0$. We show $H^{-n+1}\UU=\HH$. Let $A$ be an object in $\HH$. By $A[n-1]\in\DD^{\leq -n+1}\se\UU$, we have $A=H^{-n+1}(A[n-1])\in H^{-n+1}\UU$. 
\end{proof}

Moreover we consider the case of $n=2$, that is, the intermediate case.

\begin{proposition}\label{prop:nondeg}
If $(\DD^{\leq 0},\DD^{\geq 1})$ is nondegenerate, then every intermediate preaisle in $\DD$ is homology-determined. 
\end{proposition}

\begin{proof}
Let $\UU$ be an intermediate preaisle in $\DD$ and $X$ an object in $\DD$. Assume that $X$ belongs to $\UU$. We show $H^kX\in\UU[k]$ for any $k\in\Z$. We have $H^kX\in\UU[k]$ for any $k>0$ since it holds $H^kX=0$ by $\UU\se\DD^{\leq 0}$. Besides $H^kX\in\UU[k]$ holds for any $k<0$ since $H^kX\in\DD^{\leq 0}$ and $\DD^{\leq 0}\se\UU[k]$ hold. We show $H^0X\in\UU$. Take an exact triangle
\[
\begin{tikzcd}
  X^{\leq -1} \rar & X \rar & H^0X\rar & X^{\leq -1}[1].
\end{tikzcd}
\]
where $X^{\leq -1}$ belongs to $\DD^{\leq -1}\se\UU$. Since $\UU$ is closed under positive shifts and extensions, we have $H^0X\in\UU$. 

Assume that $X$ satisfies $H^kX\in\UU[k]$ for any $k\in\Z$. We show $X\in\UU$. For any $k>0$, it hold $H^kX\in\DD^{\geq 0}$ and $\UU[k]\se\DD^{\leq -k}\se\DD^{\leq -1}$. Therefore $H^kX=0$ holds for any $k>0$. Since $(\DD^{\leq 0},\DD^{\geq 1})$ is nondegenerate, we have $X\in\DD^{\leq 0}$. Consider the above exact triangle. Then $X^{\leq -1}\in\UU$ holds by $\DD^{\leq -1}\se\UU$ and $H^0X\in\UU$ holds by the assumption. Since $\UU$ is closed under extensions, we have $X\in\UU$. 
\end{proof}

The following is easily checked from the definition of ICE sequences. 

\begin{proposition}\label{prop:int}
There exist bijective correspondences between:
\begin{enumerate}
\item the set of ICE sequences in $\AA$ of length $2$. 
\item the set of torsion classes in $\AA$,
\end{enumerate}
The correspondence from {\rm (1)} to {\rm (2)} is given by $\{\CC(k)\}_{k\in\Z}\mapsto\CC(0)$.
\end{proposition}

The following is an analogous result of \cite[Proposition 2.1]{Woo}. 

\begin{corollary}\label{cor:HRS}
If $(\DD^{\leq 0},\DD^{\geq 1})$ is nondegenerate, then
there exist bijective correspondences between:
\begin{enumerate}
\item the set of torsion classes in $\HH$,
\item the set of intermediate preaisles in $\DD$. 
\end{enumerate}
\end{corollary}

\begin{proof}
This follows from Corollary \ref{cor:n-int} and Propositions \ref{prop:nondeg} and \ref{prop:int}.
\end{proof}

\section{A classification of $t$-structures}\label{sec:cla}

We consider the following problem: When does a preaisle $\theta(\CC)$ in Theorem \ref{thm:SR} become an aisle, that is, induce a $t$-structure on $\DD$? The following proposition gives a necessary condition of $\CC$. This is stated in \cite[Proposition 2.8]{SR} in the case of the bounded derived category of a hereditary abelian category.

\begin{proposition}\label{prop:nece}
Let $\UU$ be a homology-determined aisle in $\DD$. Then $H^k\UU$ is contravariantly finite in $\HH$ for any $k$. 
\end{proposition}

\begin{proof}
Since $\UU[k]$ is also a homology-determined aisle, it is enough to show that $H^0\UU$ is contravariantly finite in $\HH$. By Proposition \ref{prop:homdet}, it holds that $H^0\UU=\UU\cap\HH$. Let $X$ be an object in $\HH$. Since $\UU$ is an aisle in $\DD$, there exists a right $\UU$-approximation $f\colon U\to X$. Consider the exact triangle 
\[
\begin{tikzcd}
  U^{\leq -1} \rar & U \rar["\varphi"] & U^{\geq 0}\rar & U^{\leq -1}[1]
\end{tikzcd}
\]
with $U^{\leq -1}\in\DD^{\leq -1}$ and $U^{\geq 0}\in\DD^{\geq 0}$. Since $\UU$ is homology-determined, it is easy to check that $U^{\geq 0}$ belongs to $\UU$. By $X\in\DD^{\geq 0}$, there exists a morphism $g\colon U^{\geq 0}\to X$ such that $f=g\varphi$. Then $g$ is also a right $\UU$-approximation. Take the natural morphism $\psi\colon H^0U\to U^{\geq 0}$, which is a right $\HH$-approximation of $U^{\geq 0}$. Then it is easy to verify that the composition $g\psi\colon H^0U\to X$ gives a right $\UU\cap\HH$-approximation of $X$.
\end{proof}

We consider conditions which guarantee the converse of the above statement holds. 
In the rest of this section, \emph{we assume that $(\DD^{\leq 0},\DD^{\geq 1})$ is bounded}. Then $\DD$ is idempotent complete by \cite{LC}, and every contravariantly finite preaisle in $\DD$ closed under direct summands is an aisle by \cite[Corollary 4.5]{CCS}. Moreover \emph{we assume that $\HH$ is contravariantly finite in $\DD$}. Then every contravariantly finite subcategory of $\HH$ is also contravariantly finite in $\DD$. 
For example, the following $t$-structures satisfy the above conditions:
\begin{itemize}
\item \cite[Corollary. 2.12]{CPP} Algebraic $t$-structures on the bounded derived category $D^b(\mod\Lambda)$ of the category $\mod\Lambda$ consisting of finitely generated right $\Lambda$-modules over a finite dimensional algebra $\Lambda$ over a field.
\item \cite[Remark. 2.15]{CPP} The standard $t$-structure on the bounded derived category $D^b(\AA)$ of a hereditary abelian category $\AA$ which has enough injectives.
\end{itemize}

An ICE sequence $\{\CC(k)\}_{k\in\Z}$ in $\AA$ is \emph{contravariantly finite} if $\CC(k)$ is contravariantly finite in $\AA$ for any $k$. We say that an ICE sequence $\{\CC(k)\}_{k\in\Z}$ \emph{terminates to zero} provided that there is an integer $n$ such that $\CC(n)=0$.

\begin{proposition}\label{prop:term}
Let $\CC=\{\CC(k)\}_{k\in\Z}$ be a contravariantly finite ICE sequence in $\HH$ which terminates to zero. Then $\UU=\theta(\CC)$ is contravariantly finite in $\DD$. Hence $\UU$ is an aisle in $\DD$. 
\end{proposition}

\begin{proof}
We may assume that $\CC(k)=0$ holds for $k>0$. By Proposition \ref{prop:decomp} and that $\UU$ is homology-determined, the equality
\[
\UU=\bigcup_{l>0}\CC(-l)[l]*\cdots*\CC(-1)[1]*\CC(0)
\]
holds. 
By the dual statement of \cite[Theorem 1.3]{Che}, the subcategory $\CC(-l)[l]*\cdots*\CC(-1)[1]*\CC(0)$ is contravariantly finite in $\DD$ for any $l>0$. We show that any object $X$ in $\DD$ has a right $\UU$-approximation. Since $(\DD^{\leq 0},\DD^{\geq 1})$ is bounded, there is an integer $m$ such that for any $Y\in\DD^{\leq -m}$, we have $\DD(Y,X)=0$. Take a right $\CC(-m+1)[m-1]*\cdots*\CC(0)$-approximation $f\colon U\to X$ of $X$. Then we show that it gives a right $\UU$-approximation of $X$. Let $g\colon U'\to X$ be a morphism with $U'\in\UU$. By the above equality, there is an integer $l>0$ such that $U'$ belongs to $\CC(-l)[l]*\CC(-l+1)[l-1]*\cdots*\CC(0)$. If $l\leq m-1$, the morphism $g$ factors through $f$. Suppose $l\geq m$. We consider the exact triangle
\[
\begin{tikzcd}
  V \rar & U' \rar & W\rar & V[1]
\end{tikzcd}
\]
where $V\in\CC(-l)[l]*\cdots*\CC(-m)[m]$ and $W\in\CC(-m+1)[m-1]*\cdots*\CC(0)$. Then $g$ factors through $W$ by $V\in\DD^{\leq -m}$ and $\DD(\DD^{\leq -m}, X)=0$, and also $f$ since $W\in\CC(-m+1)[m-1]*\cdots*\CC(0)$ and $f$ is a right $\CC(-m+1)[m-1]*\cdots*\CC(0)$-approximation. Thus $X$ has a right $\UU$-approximation, and $\UU$ is contravariantly finite in $\DD$. 

By Remark \ref{rem:summand}, the homology-determined preaisle $\UU$ is closed under direct summands in $\DD$. Now $\UU$ is a contravariantly finite preaisle closed under direct summands. By the assumption of $(\DD^{\leq 0},\DD^{\geq 1})$, we have that $\UU$ is an aisle. 
\end{proof}

We say that an aisle $\UU$ in $\DD$ is \emph{right bounded} provided that there is an integer $n$ such that $\UU\subseteq\DD^{\leq n}$. The following result follows from Proposition \ref{prop:nece} and Proposition \ref{prop:term}. 

\begin{proposition}\label{prop:rb}
The bijections in Theorem \ref{thm:ICE} restricts to the following:
\begin{enumerate}
\item the set of contravariantly finite ICE sequences  in $\HH$ which terminate to zero,
\item the set of right bounded homology-determined aisles in $\DD$. 
\end{enumerate}
\end{proposition}

We say that an ICE sequence $\{\CC(k)\}_{k\in\Z}$ is \emph{full} provided that there are integers $m\leq n$ such that $\CC(m)=\HH$ and $\CC(n)=0$. We say that an aisle $\UU$ of $\DD$ is \emph{two-sided bounded} provided that there exist integers $m\leq n$ such that $\DD^{\leq m}\se\UU\se\DD^{\leq n}$. This notion is also called intermediate in \cite{MZ}.

\begin{proposition}\cite[Lemma 3.22]{AMY}\label{prop:alg}
Assume that $(\DD^{\leq 0},\DD^{\geq 1})$ is algebraic. Then a $t$-structure $(\UU,\VV)$ on $\DD$ is bounded if and only if $\UU$ is two-sided bounded. 
\end{proposition}

By Propositions \ref{prop:rb} and \ref{prop:alg}, we obtain the following result.

\begin{theorem}\label{thm:bdd}
The bijections in Theorem \ref{thm:ICE} restrict to the following:
\begin{enumerate}
\item the set of contravariantly finite full ICE sequences in $\HH$,
\item the set of two-sided bounded homology-determined aisles in $\DD$. 
\end{enumerate}
Moreover if $(\DD^{\leq 0},\DD^{\geq 1})$ is an algebraic $t$-structure on $\DD$, then the set {\rm (2)} coincides with
\begin{enumerate}
\item[{\rm (3)}] the set of bounded $t$-structures on $\DD$ whose aisles are homology-determined. 
\end{enumerate}
\end{theorem}

The following follows immediately from the above.

\begin{corollary}\label{cor:int}
Let $n$ be a positive integer. Then the bijections in Theorem \ref{thm:bdd} restrict to the following:
\begin{enumerate}
\item the set of contravariantly finite ICE sequences in $\HH$ of length $n$,
\item the set of $n$-intermediate $t$-structures on $\DD$ whose aisles are homology-determined. 
\end{enumerate}
\end{corollary}

Next we give a description of ICE sequences of length $n+1$ in an abelian length category $\AA$ from the viewpoint of a lattice consisting of torsion classes in $\AA$. We denote by $\tors\AA$ the set of torsion classes in $\AA$, which forms a partially ordered set by inclusion. Moreover $\tors\AA$ is a complete lattice since there are arbitrary intersections. We collect some definitions and results.

\begin{definition}
To $\TT,\UU\in\tors\AA$, we associate the set
\[
[\UU, \TT]:=\{\CC\in\tors\AA\mid\UU\se\CC\se\TT\}
\]
called an \emph{interval} in $\tors\AA$. To an interval $[\UU,\TT]$, we associate a subcategory $\HH_{[\UU,\TT]}=\TT\cap\UU^{\perp}$ called the \emph{heart} of $[\UU,\TT]$. We call an interval $[\UU,\TT]$ a \emph{wide interval} if the heart is a wide subcategory of $\AA$. We denote by $\Hasse(\tors\AA)$ the \emph{Hasse quiver} of $\tors\AA$, the quiver whose vertex set is $\tors\AA$, and there is an arrow $\TT\to\UU$ in $\tors\AA$ if and only if $\UU\subsetneq\TT$ holds and there is no $\CC\in\tors\AA$ satisfying $\UU\subsetneq\CC\subsetneq\TT$. 
\end{definition}

\begin{proposition}\label{prop:torslattice}
Let $[\UU,\TT]$ be an interval in $\tors\AA$. 
\begin{enumerate}
\item \cite[Theorem 5.2]{AP} The following conditions are equivalent:
\begin{enumerate}
\item $[\UU,\TT]$ is a wide interval.
\item $[\UU,\TT]$ is a meet interval, that is, it holds
\[
\UU=\TT\bigcap\{\CC\in[\UU,\TT]\mid \text{there is an arrow} \ \TT\to\CC \ \text{in} \ \Hasse(\tors\AA)\}.
\]
\end{enumerate}
\item \cite[Theorem 4.2]{AP} Suppose that $[\UU,\TT]$ is a wide interval. Then there are isomorphisms of lattices 
    \[
    \begin{tikzcd}[column sep = large]
      {[\UU,\TT]} \rar["{(-)\cap\UU^{\perp}}", shift left] & \tors\HH_{[\UU,\TT]} \lar["{\UU*(-)}", shift left].
    \end{tikzcd}
    \]
\item \cite[Lemma 2.9]{ES} The above bijections preserve hearts, that is, for each interval $[\UU',\TT']$ contained in $[\UU,\TT]$, we have $\HH_{[\UU',\TT']}=\HH_{[\UU'\cap\UU^{\perp},\TT'\cap\UU^{\perp}]}$ where the right hand side is considered in an abelian category $\HH_{[\UU,\TT]}$.
\end{enumerate}
\end{proposition}

Note that wide intervals in $\tors\AA$ is characterized as a lattice-theoretical property in $\tors\AA$. The operation $\alpha$ defined in Proposition \ref{prop:alpha} is understood from the viewpoint of wide intervals.

\begin{proposition}\label{prop:tors}
Let $\TT$ be a torsion class in $\AA$. Then the following statements hold.
\begin{enumerate}
\item \cite[Proposition 6.3]{AP}  $\alpha\TT$ equals to the heart of the interval $[\TT\cap {}^{\perp}\alpha\TT, \TT]$. 
\item \cite[Proposition 3.3]{ES} We set  
\[
\TT^-=\TT\bigcap\{\CC\in\tors\AA\mid \text{there is an arrow} \ \TT\to\CC \ \text{in} \ \Hasse(\tors\AA)\}.
\]
Then we have $\TT^-=\TT\cap {}^{\perp}\alpha\TT$ and $\HH_{[\TT^-,\TT]}=\alpha\TT$.
\end{enumerate}
\end{proposition}

We introduce the following notion.

\begin{definition}\label{def:max}
We call an interval of the form $[\TT^-,\TT]$ a \emph{maximal meet interval} in $\tors\AA$. More generally, we call an interval $[\UU',\TT']$ contained in a wide interval $[\UU,\TT]$ in $\tors\AA$ a \emph{maximal meet interval} in $[\UU,\TT]$ if we have
\[
\UU'=\TT'\bigcap\{\CC\in[\UU,\TT]\mid \text{there is an arrow} \ \TT'\to\CC \ \text{in} \ \Hasse(\tors\AA)\}.
\]
\end{definition}

\begin{proposition}\label{prop:meet}
Let $[\UU',\TT']$ be an interval in a wide interval $[\UU,\TT]$ in $\tors\AA$. Then $[\UU',\TT']$ is a maximal wide interval in $[\UU,\TT]$ if and only if $[\UU'\cap\UU^{\perp},\TT'\cap\UU^{\perp}]$ is a maximal meet interval in $\tors\HH_{[\UU,\TT]}$. In this case, it holds $\HH_{[\UU',\TT']}=\alpha(\TT'\cap\UU^\perp)$. 
\end{proposition}

\begin{proof}
The former part follows from a lattice isomorphism in Proposition \ref{prop:torslattice} (2). 
By Proposition \ref{prop:torslattice} (3), we have $\HH_{[\UU',\TT']}=\HH_{[\UU'\cap\UU^\perp,\TT'\cap\UU^\perp]}$. Since $[\UU'\cap\UU^\perp,\TT'\cap\UU^\perp]$ is a maximal meet interval in $\tors\HH_{[\UU,\TT]}$, we have $\HH_{[\UU'\cap\UU^\perp,\TT'\cap\UU^\perp]}=\alpha(\TT'\cap\UU^\perp)$ by Proposition \ref{prop:tors} (2). 
\end{proof}

We call a sequence $\{[\UU_k, \TT_k]\}_{k=1}^{n}$ of intervals in $\tors\AA$ a \emph{decreasing sequence of maximal meet intervals} in $\tors\AA$ provided that $[\UU_{k+1}, \TT_{k+1}]$ is a maximal meet interval in $[\UU_k, \TT_k]$ for any $k=0,\ldots, n-1$ where we set $\UU_0=0$ and $\TT_0=\AA$. We call $n$ the \emph{length} of the sequence. The following gives a description of ICE sequences via maximal meet intervals. 

\begin{theorem}\label{thm:wideint}
Let $\AA$ be an abelian length category. Then there exist bijective correspondences between
\begin{enumerate}
\item the set of decreasing sequences of maximal meet intervals in $\tors\AA$ of length $n$, 
\item the set of ICE sequences $\{\CC(k)\}_{k\in\Z}$ in $\AA$ of length $n+1$. 
\end{enumerate}
\end{theorem}

\begin{proof}
Considering the map $\{\CC(k)\}_{k\in\Z}\mapsto\{\CC(k-n)\}_{k\in\Z}$, there is a bijection between {\rm (2)} and the set {\rm (2)'} of ICE sequences $\{\CC(k)\}_{k\in\Z}$ satisfying $\CC(0)=\AA$ and $\CC(n+1)=0$. Then the maps between {\rm (1)} and {\rm (2)'} are given by Propositions \ref{prop:fullICE} and \ref{prop:dec}. It can be checked that these are mutually inverse of each other by Proposition \ref{prop:torslattice} (2). 
\end{proof}

\begin{proposition}\label{prop:fullICE}
Let $\{[\UU_k, \TT_k]\}_{k=1}^{n}$ be a decreasing sequence of maximal meet intervals in $\tors\AA$. We define subcategories of $\AA$ by
\[
    \CC(k)=\left\{ \begin{array}{ll}
      \AA & (k \leq 0) \\
      \TT_k\cap\UU_{k-1}^{\perp} & (1 \leq k\leq n) \\
      0 & (n+1\leq k) \\
    \end{array} \right. 
\]
where we put $\UU_{0}=0$ and $\TT_{0}=\AA$. Then $\{\CC(k)\}_{k\in\Z}$ is an ICE sequence in $\AA$. 
\end{proposition}

\begin{proof}
By \cite[Lemma 2.2]{ES}, it is enough to show the following claim: $\CC(k+1)$ is a torsion class in $\alpha(\CC(k))$ for $k=0,\ldots, n-1$. Clearly, the claim holds for $k=0$. Suppose $k>0$. By $\TT_{k+1}\se\TT_k$ and $\UU_k\se\UU_{k+1}$, it holds $\CC(k+1)\se\CC(k)$. By $\TT_{k+1}\in[\UU_k,\TT_k]$ and Proposition \ref{prop:torslattice} (2), the subcategory $\CC(k+1)=\TT_{k+1}\cap\UU_{k}^{\perp}$ is a torsion class in $\HH_{[\UU_k,\TT_k]}$. By Proposition \ref{prop:meet}, we have
\[
\HH_{[\UU_k,\TT_k]}=\alpha(\TT_k\cap\UU_{k-1}^\perp)=\alpha(\CC(k)).
\]
Thus $\CC(k+1)$ is a torsion class in $\alpha(\CC(k))$.
\end{proof}

\begin{proposition}\label{prop:con}
Let $[\UU,\TT]$ be a wide interval in $\tors\AA$ and $\CC$ a torsion class in $\HH_{[\UU,\TT]}$. Then there is a maximal meet interval $[\UU',\TT']$ in $[\UU,\TT]$ such that $\TT'\cap\UU^\perp=\CC$ and $\HH_{[\UU',\TT']}=\alpha\CC$.
\end{proposition}

\begin{proof}
Take a maximal meet interval $[\CC^-,\CC]$ in $\tors\HH_{[\UU,\TT]}$. By Proposition \ref{prop:torslattice} (2), there is a maximal meet interval $[\UU',\TT']$ in $[\UU,\TT]$ such that $\TT'\cap\UU^\perp=\CC$. By Proposition \ref{prop:torslattice} (3), we have $\HH_{[\UU',\TT']}=\HH_{[\CC^-,\CC]}=\alpha\CC$. 
\end{proof}

\begin{proposition}\label{prop:dec}
Let $\{\CC(k)\}_{k\in\Z}$ be an ICE sequence in $\AA$ satisfying $\CC(0)=\AA$ and $\CC(n+1)=0$. Then there is a decreasing sequence of maximal meet intervals $\{[\UU_k,\TT_k]\}_{k=1}^{n}$ in $\tors\AA$ such that $\TT_{k}\cap\UU_{k-1}^\perp=\CC(k)$ and $\HH_{[\UU_k,\TT_k]}=\alpha(\CC(k))$ for any $k$ where $\UU_0=0$.  
\end{proposition}

\begin{proof}
We construct $\{[\UU_k,\TT_k]\}_{k=1}^{n}$ inductively. Put $\TT_1=\CC(1)$ and $\UU_1=\CC(1)^-$. Suppose that we have $[\UU_k,\TT_k]$ such that $\HH_{[\UU_k,\TT_k]}=\alpha(\CC(k))$. Applying Proposition \ref{prop:con} to a wide interval $[\UU_k,\TT_k]$ and a torsion class $\CC(k+1)$ in $\HH_{[\UU_k,\TT_k]}$, we obtain a maximal wide interval $[\UU_{k+1},\TT_{k+1}]$ in $[\UU_k,\TT_k]$ such that $\TT_{k+1}\cap\UU_k^\perp=\CC(k+1)$ and $\HH_{[\UU_{k+1},\TT_{k+1}]}=\alpha(\CC(k+1))$.
\end{proof}

Now we establish a classification of $(n+1)$-intermediate $t$-structures on $D^b(\mod\Lambda)$ whose aisles are homology-determined for a $\tau$-tilting finite algebra $\Lambda$. A finite dimensional algebra $\Lambda$ is \emph{$\tau$-tilting finite} if every torsion class in $\mod\Lambda$ is functorially finite, see \cite[Theorem 3.8]{DIJ}. Then, by \cite[Proposition 4.20]{ES}, every ICE-closed subcategory of $\mod\Lambda$ is functorially finite, in particular contravariantly finite. For any positive integer $n$, we obtain a classification of $(n+1)$-intermediate $t$-structures whose aisles are homology-determined via ICE sequences and the lattice of torsion classes:

\begin{theorem}\label{thm:class}
Let $\Lambda$ be a $\tau$-tilting finite algebra and $\tors\Lambda$ the lattice consisting of torsion classes in $\mod\Lambda$. Then there are one-to-one correspondences between
\begin{enumerate}
\item the set of $(n+1)$-intermediate $t$-structures on $D^b(\mod\Lambda)$ whose aisles are homology-determined,
\item the set of ICE sequences in $\mod\Lambda$ of length $n+1$,
\item the set of decreasing sequences of maximal meet intervals in $\tors\Lambda$ of length $n$,
\end{enumerate}
\end{theorem}

\begin{proof}
This follows from Corollary \ref{cor:int} and Theorem \ref{thm:wideint}.
\end{proof}

By Proposition \ref{prop:alg}, it follows that for any bounded $t$-structure $(\UU,\VV)$ on $D^b(\mod\Lambda)$, there exist $n,m\in\Z$ such that $(\UU[m],\VV[m])$ is $(n+1)$-intermediate $t$-structure. Hence we can deal with all bounded $t$-structures on $D^b(\mod\Lambda)$ whose aisles are homology-determined by using decreasing sequences of maximal meet intervals in $\tors\Lambda$. 

\begin{remark}
The following conditions are equivalent for a finite dimensional algebra $\Lambda$.
\begin{enumerate}
\item Every ICE-closed subcategory of $\mod\Lambda$ is contravariantly finite.
\item Every wide subcategory of $\mod\Lambda$ is functorially finite. 
\end{enumerate}
This follows easily from \cite[Proposition 4.12]{Eno} and \cite[Corollary 3.5]{ES}. Theorem \ref{thm:class} holds under the condition (1) even without the $\tau$-tilting finiteness of $\Lambda$. On the other hand, the condition (2) is considered in \cite[Questions 5.11]{AS}. We do not know whether the above conditions are equivalent to that $\Lambda$ is $\tau$-tilting finite.
\end{remark}

\begin{example}
Let $K$ be a field and $KQ$ a path algebra for $Q\colon 1\ot 2\ot 3$.   The Auslander-Reiten quiver of $\mod KQ$ is as follows.
  
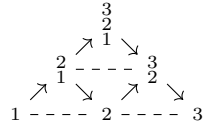
\[
  \begin{tikzpicture}[scale=.6]
   \node (1) at (0,0) {$\sst{1}$};
   \node (21) at (1,1) {$\sst{2 \\ 1}$};
   \node (2) at (2,0) {$\sst{2}$};
   \node (32) at (3,1) {$\sst{3 \\ 2}$};
   \node (3) at (4,0) {$\sst{3}$};
   \node (321) at (2,2) {$\sst{3 \\ 2 \\ 1}$};

   \draw[->] (1) -- (21);
   \draw[->] (21) -- (2);
   \draw[->] (2) -- (32);
   \draw[->] (32) -- (3);
   \draw[->] (21) -- (321);
   \draw[->] (321) -- (32);

   \draw[dashed] (1) -- (2);
   \draw[dashed] (2) -- (3);
   \draw[dashed] (21) -- (32);
  \end{tikzpicture}
  \]
The Hasse quiver $\Hasse(\tors KQ)$ is depicted in Figure \ref{hasseA3}, where the vertices are the corresponding support $\tau$-tilting modules, see \cite{AIR}. 

  \begin{figure}[htp]
    \begin{tikzpicture}
      \node[block=3] (121321) at (4,6) {\nodepart{one}$\sst{1}$\nodepart{two}$\sst{2 \\ 1}$\nodepart{three}$\sst{3 \\ 2 \\ 1}$};
      \node[block=2] (121) at (6,5) {\nodepart{one}$\sst{1}$\nodepart{two}$\sst{2 \\ 1}$};
      \node[block=3] (221321) at (2,5.5) {\nodepart{one}$\sst{2}$\nodepart{two}$\sst{2 \\ 1}$\nodepart{three}$\sst{3 \\ 2 \\ 1}$};
      \node[block=3] (13321) at (2,4) {\nodepart{one}$\sst{1}$\nodepart{two}$\sst{3}$\nodepart{three}$\sst{3 \\ 2 \\ 1}$};
      \node[block=2] (221) at (4,4) {\nodepart{one}$\sst{2}$\nodepart{two}$\sst{2 \\ 1}$};
      \node[block=3] (232321) at (0,4.5) {\nodepart{one}$\sst{2}$\nodepart{two}$\sst{3 \\ 2}$\nodepart{three}$\sst{3 \\ 2 \\ 1}$};
      \node[block=2] (232) at (2,3) {\nodepart{one}$\sst{2}$\nodepart{two}$\sst{3 \\ 2}$};
      \node[block=3] (332321) at (0,3) {\nodepart{one}$\sst{3}$\nodepart{two}$\sst{3 \\ 2}$\nodepart{three}$\sst{3 \\ 2 \\ 1}$};
      \node[block=2] (13) at (4.5,2) {\nodepart{one}$\sst{1}$\nodepart{two}$\sst{3}$};
      \node[rectangle,draw] (2) at (3,2) {\nodepart{one}$\sst{2}$};
      \node[block=2] (332) at (1,2) {\nodepart{one}$\sst{3}$\nodepart{two}$\sst{3 \\ 2}$};
      \node[rectangle,draw] (3) at (2,1) {\nodepart{one}$\sst{3}$};
      \node[rectangle,draw] (1) at (6,1) {\nodepart{one}$\sst{1}$};
      \node (0) at (4,0) {$\sst{0}$};

      \draw[->] (121321) -- (121);
      \draw[->] (121321) -- (221321);
      \draw[->] (121321) -- (13321);
      \draw[->] (121) -- (221);
      \draw[->] (121) -- (1);
      \draw[->] (221321) -- (221);
      \draw[->] (221321) -- (232321);
      \draw[->] (13321) -- (332321);
      \draw[->] (13321) -- (13);
      \draw[->] (221) -- (2);
      \draw[->] (232321) -- (232);
      \draw[->] (232321) -- (332321);
      \draw[->] (232) -- (332);
      \draw[->] (232) -- (2);
      \draw[->] (332321) -- (332);
      \draw[->] (13) -- (1);
      \draw[->] (13) -- (3);
      \draw[->] (332) -- (3);
      \draw[->] (1) -- (0);
      \draw[->] (2) -- (0);
      \draw[->] (3) -- (0);
    \end{tikzpicture}
    \caption{$\Hasse(\tors KQ)$}
    \label{hasseA3}
  \end{figure}

For a support $\tau$-tilting module $T$, we denote by $\Fac T$ the torsion class corresponding to $T$. Put $\TT_1=\Fac(\sst{2}\oplus\sst{2 \\ 1}\oplus\sst{3 \\ 2 \\ 1}), \UU_1=\Fac(\sst{2}), \TT_2=\Fac(\sst{2}\oplus\sst{3 \\ 2}\oplus\sst{3 \\ 2 \\ 1})$ and $\UU_2=\Fac(\sst{2}\oplus\sst{3 \\ 2})$. Then $\{[\UU_k,\TT_k]\}_{k=1}^{2}$ is a decreasing sequence of maximal meet intervals of length $2$. The corresponding ICE sequence of length $3$ is as follows:

  
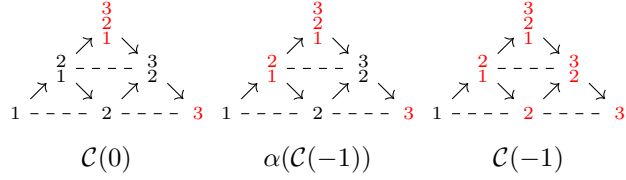
\[
  \begin{tikzpicture}[scale=.6]
   \node (1) at (0,0) {$\sst{1}$};
   \node (21) at (1,1) {$\sst{2 \\ 1}$};
   \node (2) at (2,0) {$\sst{2}$};
   \node (32) at (3,1) {$\sst{3 \\ 2}$};
   \node[red] (3) at (4,0) {$\sst{3}$};
   \node[red] (321) at (2,2) {$\sst{3 \\ 2 \\ 1}$};
   \node (c) at (2,-1) {$\CC(0)$};

   \draw[->] (1) -- (21);
   \draw[->] (21) -- (2);
   \draw[->] (2) -- (32);
   \draw[->] (32) -- (3);
   \draw[->] (21) -- (321);
   \draw[->] (321) -- (32);

   \draw[dashed] (1) -- (2);
   \draw[dashed] (2) -- (3);
   \draw[dashed] (21) -- (32);
  \end{tikzpicture}
  \begin{tikzpicture}[scale=.6]
   \node (1) at (0,0) {$\sst{1}$};
   \node[red] (21) at (1,1) {$\sst{2 \\ 1}$};
   \node (2) at (2,0) {$\sst{2}$};
   \node (32) at (3,1) {$\sst{3 \\ 2}$};
   \node[red] (3) at (4,0) {$\sst{3}$};
   \node[red] (321) at (2,2) {$\sst{3 \\ 2 \\ 1}$};
   \node (c) at (2,-1) {$\alpha(\CC(-1))$};

   \draw[->] (1) -- (21);
   \draw[->] (21) -- (2);
   \draw[->] (2) -- (32);
   \draw[->] (32) -- (3);
   \draw[->] (21) -- (321);
   \draw[->] (321) -- (32);

   \draw[dashed] (1) -- (2);
   \draw[dashed] (2) -- (3);
   \draw[dashed] (21) -- (32);
  \end{tikzpicture}
  \begin{tikzpicture}[scale=.6]
   \node (1) at (0,0) {$\sst{1}$};
   \node[red] (21) at (1,1) {$\sst{2 \\ 1}$};
   \node[red] (2) at (2,0) {$\sst{2}$};
   \node[red] (32) at (3,1) {$\sst{3 \\ 2}$};
   \node[red] (3) at (4,0) {$\sst{3}$};
   \node[red] (321) at (2,2) {$\sst{3 \\ 2 \\ 1}$};
   \node (c) at (2,-1) {$\CC(-1)$};

   \draw[->] (1) -- (21);
   \draw[->] (21) -- (2);
   \draw[->] (2) -- (32);
   \draw[->] (32) -- (3);
   \draw[->] (21) -- (321);
   \draw[->] (321) -- (32);

   \draw[dashed] (1) -- (2);
   \draw[dashed] (2) -- (3);
   \draw[dashed] (21) -- (32);
  \end{tikzpicture}
  \]

The induced aisle is depicted as red vertices in the following:

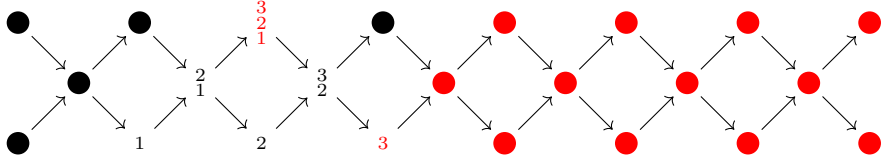
\[
  \begin{tikzpicture}[scale=.8]
   \node[blackv] (20) at (2,0) {};
   \node[blackv] (22) at (2,2) {};
   \node[blackv] (31) at (3,1) {};   
   \node[black] (40) at (4,0) {$\sst{1}$};
   \node[blackv] (42) at (4,2) {};
   \node[black] (51) at (5,1) {$\sst{2 \\ 1}$};
   \node[black] (60) at (6,0) {$\sst{2}$};
   \node[red] (62) at (6,2) {$\sst{3 \\ 2 \\ 1}$};
   \node[black] (71) at (7,1) {$\sst{3 \\ 2}$};
   \node[red] (80) at (8,0) {$\sst{3}$};
   \node[blackv] (82) at (8,2) {};   
   \node[redv] (91) at (9,1) {};
   \node[redv] (100) at (10,0) {};
   \node[redv] (102) at (10,2) {};
   \node[redv] (111) at (11,1) {};
   \node[redv] (120) at (12,0) {};
   \node[redv] (122) at (12,2) {};
   \node[redv] (131) at (13,1) {};
   \node[redv] (140) at (14,0) {};
   \node[redv] (142) at (14,2) {};
   \node[redv] (151) at (15,1) {};
   \node[redv] (160) at (16,0) {};
   \node[redv] (162) at (16,2) {};

   \draw[->] (20) -- (31);
   \draw[->] (22) -- (31);
   \draw[->] (31) -- (40);
   \draw[->] (31) -- (42);
   \draw[->] (40) -- (51);
   \draw[->] (42) -- (51);
   \draw[->] (51) -- (60);
   \draw[->] (51) -- (62);
   \draw[->] (60) -- (71);
   \draw[->] (62) -- (71);
   \draw[->] (71) -- (80);
   \draw[->] (71) -- (82);
   \draw[->] (80) -- (91);
   \draw[->] (82) -- (91);
   \draw[->] (91) -- (100);
   \draw[->] (91) -- (102);
   \draw[->] (100) -- (111);
   \draw[->] (102) -- (111);
   \draw[->] (111) -- (120);
   \draw[->] (111) -- (122);
   \draw[->] (120) -- (131);
   \draw[->] (122) -- (131);
   \draw[->] (131) -- (140);
   \draw[->] (131) -- (142);
   \draw[->] (140) -- (151);
   \draw[->] (142) -- (151);
   \draw[->] (151) -- (160);
   \draw[->] (151) -- (162);

  \end{tikzpicture}
\]
\end{example}

\begin{example}
Consider the following algebra
\[
\Lambda : = K [ 1 \xleftarrow{\be} 2 \xleftarrow{\al} 3]/(\al \be).
\]
  The Auslander-Reiten quiver of $\mod \Lambda$ is as follows.
  
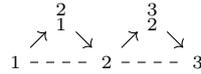
\[
  \begin{tikzpicture}[scale=.6]
   \node (1) at (0,0) {$\sst{1}$};
   \node (21) at (1,1) {$\sst{2 \\ 1}$};
   \node (2) at (2,0) {$\sst{2}$};
   \node (32) at (3,1) {$\sst{3 \\ 2}$};
   \node (3) at (4,0) {$\sst{3}$};

   \draw[->] (1) -- (21);
   \draw[->] (21) -- (2);
   \draw[->] (2) -- (32);
   \draw[->] (32) -- (3);

   \draw[dashed] (1) -- (2);
   \draw[dashed] (2) -- (3);
  \end{tikzpicture}
  \]
The algebra $\Lambda$ is derived equivalent to $KQ$ since $\sst{1}\oplus\sst{3}\oplus\sst{3 \\ 2 \\1}$ is a tilting $KQ$-module and its endomorphism algebra is isomorphic to $\Lambda$. Thus the Auslander-Reiten quiver of $D^b(\mod\Lambda)$ is isomorphic to that of $D^b(\mod KQ)$ and depicted as follows. 

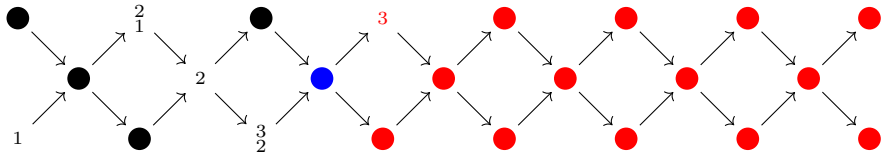
\[
  \begin{tikzpicture}[scale=.8]
   \node (11) at (1.8,1) {};
   \node[black] (20) at (2,0) {$\sst{1}$};
   \node[blackv] (22) at (2,2) {};
   \node[blackv] (31) at (3,1) {};   
   \node[blackv] (40) at (4,0) {};
   \node[black] (42) at (4,2) {$\sst{2 \\ 1}$};
   \node[black] (51) at (5,1) {$\sst{2}$};
   \node[black] (60) at (6,0) {$\sst{3 \\ 2}$};
   \node[blackv] (62) at (6,2) {};
   \node[bluev] (71) at (7,1) {};
   \node[redv] (80) at (8,0) {};
   \node[red] (82) at (8,2) {$\sst{3}$};   
   \node[redv] (91) at (9,1) {};
   \node[redv] (100) at (10,0) {};
   \node[redv] (102) at (10,2) {};
   \node[redv] (111) at (11,1) {};
   \node[redv] (120) at (12,0) {};
   \node[redv] (122) at (12,2) {};
   \node[redv] (131) at (13,1) {};   
   \node[redv] (140) at (14,0) {};
   \node[redv] (142) at (14,2) {};
   \node[redv] (151) at (15,1) {};
   \node[redv] (160) at (16,0) {};
   \node[redv] (162) at (16,2) {};

   \draw[->] (20) -- (31);
   \draw[->] (22) -- (31);
   \draw[->] (31) -- (40);
   \draw[->] (31) -- (42);
   \draw[->] (40) -- (51);
   \draw[->] (42) -- (51);
   \draw[->] (51) -- (60);
   \draw[->] (51) -- (62);
   \draw[->] (60) -- (71);
   \draw[->] (62) -- (71);
   \draw[->] (71) -- (80);
   \draw[->] (71) -- (82);
   \draw[->] (80) -- (91);
   \draw[->] (82) -- (91);
   \draw[->] (91) -- (100);
   \draw[->] (91) -- (102);
   \draw[->] (100) -- (111);
   \draw[->] (102) -- (111);
   \draw[->] (111) -- (120);
   \draw[->] (111) -- (122);
   \draw[->] (120) -- (131);
   \draw[->] (122) -- (131);
   \draw[->] (131) -- (140);
   \draw[->] (131) -- (142);
   \draw[->] (140) -- (151);
   \draw[->] (142) -- (151);
   \draw[->] (151) -- (160);
   \draw[->] (151) -- (162);

  \end{tikzpicture}
\]

The Hasse quiver $\Hasse(\tors\Lambda)$ is depicted in Figure \ref{fig:naktors}. Put $\TT_1=\Fac(\sst{2}\oplus\sst{2 \\ 1}\oplus\sst{3 \\ 2}), \UU_1=\Fac(\sst{2}), \TT_2=\Fac(\sst{2}\oplus\sst{3 \\ 2})$ and $\UU_2=\Fac(\sst{2})$. Then $\{[\UU_k,\TT_k]\}$ is a decreasing sequence of maximal meet intervals of length $2$. The corresponding ICE sequence of length $3$ is as follows:

  
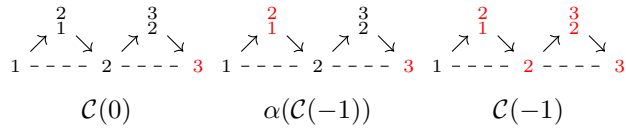
\[
  \begin{tikzpicture}[scale=.6]
   \node (1) at (0,0) {$\sst{1}$};
   \node (21) at (1,1) {$\sst{2 \\ 1}$};
   \node (2) at (2,0) {$\sst{2}$};
   \node (32) at (3,1) {$\sst{3 \\ 2}$};
   \node[red] (3) at (4,0) {$\sst{3}$};
   \node (c) at (2,-1) {$\CC(0)$};
   
   \draw[->] (1) -- (21);
   \draw[->] (21) -- (2);
   \draw[->] (2) -- (32);
   \draw[->] (32) -- (3);

   \draw[dashed] (1) -- (2);
   \draw[dashed] (2) -- (3);
  \end{tikzpicture}
  \begin{tikzpicture}[scale=.6]
   \node (1) at (0,0) {$\sst{1}$};
   \node[red] (21) at (1,1) {$\sst{2 \\ 1}$};
   \node (2) at (2,0) {$\sst{2}$};
   \node (32) at (3,1) {$\sst{3 \\ 2}$};
   \node[red] (3) at (4,0) {$\sst{3}$};
   \node (c) at (2,-1) {$\alpha(\CC(-1))$};
   
   \draw[->] (1) -- (21);
   \draw[->] (21) -- (2);
   \draw[->] (2) -- (32);
   \draw[->] (32) -- (3);

   \draw[dashed] (1) -- (2);
   \draw[dashed] (2) -- (3);
  \end{tikzpicture}
  \begin{tikzpicture}[scale=.6]
   \node (1) at (0,0) {$\sst{1}$};
   \node[red] (21) at (1,1) {$\sst{2 \\ 1}$};
   \node[red] (2) at (2,0) {$\sst{2}$};
   \node[red] (32) at (3,1) {$\sst{3 \\ 2}$};
   \node[red] (3) at (4,0) {$\sst{3}$};
   \node (c) at (2,-1) {$\CC(-1)$};
   
   \draw[->] (1) -- (21);
   \draw[->] (21) -- (2);
   \draw[->] (2) -- (32);
   \draw[->] (32) -- (3);

   \draw[dashed] (1) -- (2);
   \draw[dashed] (2) -- (3);
  \end{tikzpicture}
  \]
  
  The induced aisle consists of the red vertices.

  \begin{figure}[htp]
    \begin{tikzpicture}[scale = .6]
      \node (0) at (0,0) {0};
      \node[rectangle,draw] (1) at (4,2) {$\sst{1}$};
      \node[rectangle,draw] (2) at (-4,2) {$\sst{2}$};
      \node[rectangle,draw] (3) at (0,2) {$\sst{3}$};
      \node[block=2] (323) at (-3,4.5) {\nodepart{one}$\sst{3 \\ 2}$\nodepart{two}$\sst{3}$};
      \node[block=2] (221) at (0,5) {\nodepart{one}$\sst{2}$\nodepart{two}$\sst{2 \\ 1}$};
      \node[block=2] (13) at (2,4) {\nodepart{one}$\sst{1}$\nodepart{two}$\sst{3}$};
      \node[block=2] (232) at (-6,6) {\nodepart{one}$\sst{2}$\nodepart{two}$\sst{3 \\ 2}$};
      \node[block=3] (22132) at (-4,8) {\nodepart{one}$\sst{2}$\nodepart{two}$\sst{2 \\ 1}$\nodepart{three}$\sst{3 \\ 2}$};
      \node[block=3] (1323) at (0,8) {\nodepart{one}$\sst{1}$\nodepart{two}$\sst{3 \\ 2}$\nodepart{three}$\sst{3}$};
      \node[block=2] (121) at (4,8) {\nodepart{one}$\sst{1}$\nodepart{two}$\sst{2 \\ 1}$};
      \node[block=3] (12132) at (0,10) {\nodepart{one}$\sst{1}$\nodepart{two}$\sst{2 \\ 1}$\nodepart{three}$\sst{3 \\ 2}$};

      \draw[->] (12132) -- (22132);
      \draw[->] (12132) -- (1323);
      \draw[->] (12132) -- (121);
      \draw[->] (22132) -- (232);
      \draw[->] (22132) -- (221);
      \draw[->] (1323) -- (323);
      \draw[->] (1323) -- (13);
      \draw[->] (121) -- (221);
      \draw[->] (121) -- (1);
      \draw[->] (232) -- (323);
      \draw[->] (232) -- (2);
      \draw[->] (221) -- (2);
      \draw[->] (323) -- (3);
      \draw[->] (13) -- (1);
      \draw[->] (13) -- (3);
      \draw[->] (2) -- (0);
      \draw[->] (3) -- (0);
      \draw[->] (1) -- (0);
    \end{tikzpicture}
    \caption{$\Hasse(\tors \Lambda)$}
    \label{fig:naktors}
  \end{figure}
\end{example}

\begin{remark}\label{rem:coaisle}
Note that the coaisle consisting of the black and blue vertices in the previous example is not homology-determined. Indeed, the object $X$ which corresponds to the blue vertex belongs to the coaisle, but the 0-th cohomology $H^0X$ is isomorphic to $\sst{3}$ and belongs to the aisle, not the coaisle.
\end{remark}

\begin{ack}
The author would like to thank his supervisor Hiroyuki Nakaoka for his helpful discussion. The author would like to thank Osamu Iyama, Haruhisa Enomoto, Yasuaki Gyoda and Syunya Saito for their useful comments. The author is supported by JSPS KAKENHI Grant Number JP22J20611.
\end{ack}


\begin{thebibliography}{DIRRT}

  \bibitem[AIR]{AIR}
  T. Adachi, O. Iyama, I. Reiten,
  \emph{$\tau$-tilting theory},
  Compos. Math. 150 (2014), no. 3, 415--452.
  
  \bibitem[AMY]{AMY}
  T. Adachi, Y. Mizuno, D. Yang,
  \emph{Discreteness of silting objects and $t$-structures in triangulated categories},
  Proc. Lond. Math. Soc. (3) 118 (2019), no.1, 1--42.

  \bibitem[AJS]{AJS}
  L. Alonso Tarr\'{i}o, A. Jerem\'{i}as L\'{o}pez, M. Saor\'{i}n,
  \emph{Compactly generated $t$-structures on the derived category of a Noetherian ring},
  J. Algebra 324 (2010), no. 3, 313--346. 
  
  \bibitem[AH]{AH}
  L. Angeleri H\"{u}gel, M. Hrbek,
  \emph{Parametrizing torsion pairs in derived categories},
  Represent. Theory 25 (2021), 679--731.
  
  \bibitem[AMV]{AMV}
  L. Angeleri H\"{u}gel, F. Marks, J. Vit\'{o}ria,
  \emph{Silting modules},
  Int. Math. Res. Not. IMRN (2016), no. 4, 1251--1284.
  
  \bibitem[AS]{AS}
  L. Angeleri H\"{u}gel, F. Sentieri,
  \emph{Wide coreflective subcategories and torsion pairs},
  arXiv:2304.00845.

  \bibitem[AP]{AP}
  S. Asai, C. Pfeifer,
  \emph{Wide subcategories and lattices of torsion classes},
  Algebr. Represent. Theory 25 (2022), no. 6, 1611--1629.

  \bibitem[BBD]{BBD}
  A. A. Beilinson, J. Bernstein, P. Deligne,
  \emph{Faisceaux pervers},
  Analysis and topology on singular spaces, I (Luminy, 1981), 5--171,
  Ast\`{e}risque, 100, Soc. Math. France, Paris, 1982.
  
  \bibitem[BR]{BR}
  A. Beligiannis, I. Reiten,
  \emph{Homological and Homotopical Aspects of Torsion Theories},
  Mem. Amer. Math. Soc. 188 (2007), no.883, viii+207 pp. 
  
  \bibitem[BB]{BB}
  S. Brenner, M.C.R. Butler,
  \emph{Generalizations of the Bernstein-Gelfand-Ponomarev reflection functors},
  Representation theory, II (Proc. Second Internat. Conf., Carleton Univ., Ottawa, Ont., 1979), pp. 103--169,
  Lecture Notes in Math., 832, Springer, Berlin-New York, 1980.
  
  \bibitem[Br\"{u}]{Bru}
  K. Br\"{u}ning,
  \emph{Thick subcategories of the derived category of a hereditary algebra},
  Homology Homotopy Appl. 9 (2007), no. 2, 165--176.
  
  \bibitem[Che]{Che}
  X.-W. Chen,
  \emph{Extensions of covariantly finite subcategories},
  Arch. Math. (Basel) 93 (2009), no. 1, 29--35.
  
  \bibitem[CPP]{CPP}
  R. Coelho Sim\~{o}es, D. Pauksztello, D. Ploog,
  \emph{Functorially finite hearts, simple-minded systems in negative cluster categories, and noncrossing partitions},
  Compos. Math. 158 (2022), no. 1, 211--243.
  
  \bibitem[CCS]{CCS}
  M. Cort\'{e}s-Izurdiaga, S. Crivei, M. Saor\'{i}n, 
  \emph{Reflective and coreflective subcategories},
  J. Pure Appl. Algebra 227 (2023), no. 5, Paper No. 107267, 43 pp.
  
  \bibitem[Dic]{Dic}
  S.E. Dickson,
  \emph{A torsion theory for Abelian categories},
  Trans. Amer. Math. Soc. 121 (1966), 223--235.
  
  \bibitem[DIJ]{DIJ}
  L. Demonet, O. Iyama, G. Jasso,
  \emph{$\tau$-tilting finite algebras, bricks, and $g$-vectors},
  Int. Math. Res. Not. IMRN 2019, no. 3, 852--892.
  
  \bibitem[DIRRT]{DIRRT}
  L. Demonet, O. Iyama, N. Reading, I. Reiten, H. Thomas,
  \emph{Lattice theory of torsion classes: beyond $\tau$-tilting theory},
  Trans. Amer. Math. Soc. Ser. B10 (2023), 542--612.
  
  \bibitem[Eno]{Eno}
  H. Enomoto,
  \emph{Rigid modules and ICE-closed subcategories in quiver representations},
  J. Algebra 594 (2022), 364--388.
  
  \bibitem[Eno2]{Eno2}
  H. Enomoto,
  \emph{From the lattice of torsion classes to the posets of wide
subcategories and ICE-closed subcategories},
  Algebr. Represent. Theory (2023). \url{https://doi.org/10.1007/s10468-023-10214-0}

  \bibitem[ES]{ES}
  H. Enomoto, A. Sakai,
  \emph{ICE-closed subcategories and wide $\tau$-tilting modules},
  Math. Z. 300 (2022), no. 1, 541--577.
  
  \bibitem[GKR]{GKR}
  A. L. Gorodentsev, S. A. Kuleshov, A. N. Rudakov,
  \emph{$t$-stabilities and $t$-structures on triangulated categories},
  Izv. Math. 68 (2004), no. 4, 749--781.
  
  \bibitem[HRS]{HRS}
  D. Happel, I. Reiten, S.O. Smal\o,
  \emph{Tilting in abelian categories and quasitilted algebras},
  Mem. Amer. Math. Soc. 120 (1996), no. 575, viii+ 88 pp.
  
  \bibitem[Hrb]{Hrb}
  M. Hrbek, 
  \emph{Compactly generated $t$-structures in the derived category of a commutative ring},
  Math. Z. 295 (2020), no. 1--2, 47--72.
  
  \bibitem[IT]{IT}
  C. Ingalls, H. Thomas,
  \emph{Noncrossing partitions and representations of quivers},
  Compos. Math. 145 (2009), no. 6, 1533--1562.
  
  \bibitem[KV]{KV}
  B. Keller, D. Vossieck,
  \emph{Aisles in derived categories},
  Bull. Soc. Math. Belg. S\'{e}r. A 40 (1988), no. 2, 239--253.
  
  \bibitem[KY]{KY}
  S. Koenig, D. Yang,
  \emph{Silting objects, simple-minded collections, $t$-structures and co-$t$-structures for finite-dimensional algebras},
  Doc. Math. 19 (2014), 403--438.
  
  \bibitem[LC]{LC}
  J. Le, X.-W. Chen,
  \emph{Karoubianness of a triangulated category},
  J. Algebra 310 (2007), no. 1, 452--457.
  
  \bibitem[MS]{MS}
  F. Marks, J. \v{S}t\!'ov\'{i}\v{c}ek,
  \emph{Torsion classes, wide subcategories and localisations},
  Bull. London Math. Soc. 49 (2017), Issue 3, 405--416.
  
  \bibitem[MZ]{MZ}
  F. Marks, A. Zvonareva,
  \emph{Lifting and restricting $t$-structures},
  Bull. Lond. Math. Soc. 55 (2023), no. 2, 640--657.
  
  \bibitem[ST]{ST}
  M. J. Souto Salorio, S. Trepode, 
  \emph{$T$-structures on the bounded derived category of the Kronecker algebra},
  Appl. Categ. Structures 20 (2012), no. 5, 513--529.
  
  \bibitem[SR]{SR}
  D. Stanley, A.-C. van Roosmalen,
  \emph{$t$-structures on hereditary categories}, 
  Math. Z. 293 (2019), no. 1--2, 731--766.
  
  \bibitem[Woo]{Woo}
  J. Woolf,
  \emph{Stability conditions, torsion theories and tilting},
  J. Lond. Math. Soc. (2) 82 (2010), no.3, 663--682.
  
  \bibitem[ZC]{ZC}
  C. Zhang, H. Cai,
  \emph{A note on thick subcategories and wide subcategories},
  Homology Homotopy Appl. 19 (2017), no. 2, 131--139.

\end{thebibliography}
\end{document}